\documentclass[11pt]{amsart}

\usepackage[colorlinks=true, pdfstartview=FitV, linkcolor=blue, citecolor=blue, urlcolor=blue]{hyperref}

\usepackage{geometry}                
\usepackage{graphicx}
\usepackage{amssymb}
\usepackage{epstopdf}
\usepackage{lipsum}
\usepackage{color}
\usepackage[utf8]{inputenc}

\definecolor{darkblue}{rgb}{0.0,0,0.7} 
\definecolor{lightblue}{rgb}{0.4,0.4,1.0} 
\definecolor{darkred}{rgb}{0.7,0,0} 
\definecolor{darkgreen}{rgb}{0, .6, 0} 

\newtheorem{theorem}{Theorem}
\numberwithin{theorem}{section}
\newtheorem{prop}[theorem]{Proposition}
\newtheorem{cor}[theorem]{Corollary}
\newtheorem{lemma}[theorem]{Lemma}

\theoremstyle{definition}
\newtheorem{definition}[theorem]{Definition}
\newtheorem{example}[theorem]{Example}
\newtheorem{remark}[theorem]{Remark}

\def\la{{\lambda}}
\def\al{{\alpha}}
\def\be{{\beta}}

\def\CC{{\mathbb C}}

\def\mt{{\tilde m}}

\def\bT{{\mathcal T}}

\def\oT{{\overline\bT}}

\def\mvdash{{\,\vdash\!\!\vdash}}

\def\dcl{{\{\!\!\{}}
\def\dcr{{\}\!\!\}}}
\def\o{\overline}

\def\MSP{{\rm MSP}}

\newcommand{\defn}[1]{{\color{darkred}\emph{#1}}} 

\newcommand{\pchoose}[2]{\begin{pmatrix}#1\\ #2\end{pmatrix}}

\newdimen\squaresize \squaresize=10pt
\newdimen\thickness \thickness=0.4pt

\def\square#1{\hbox{\vrule width \thickness
     \vbox to \squaresize{\hrule height \thickness\vss
        \hbox to \squaresize{\hss#1\hss}
     \vss\hrule height\thickness}
\unskip\vrule width \thickness}
\kern-\thickness}

\def\vsquare#1{\vbox{\square{$#1$}}\kern-\thickness}

\def\young#1{
\vbox{\smallskip\offinterlineskip
\halign{&\vsquare{##}\cr #1}}}

\def\thisbox#1{\kern-.09ex\fbox{#1}}
\def\downbox#1{\lower1.200em\hbox{#1}}
\newdimen\Squaresize \Squaresize=15pt
\newdimen\Thickness \Thickness=0.4pt

\def\Square#1{\hbox{\vrule width \Thickness
     \vbox to \Squaresize{\hrule height \Thickness\vss
        \hbox to \Squaresize{\hss#1\hss}
     \vss\hrule height\Thickness}
\unskip\vrule width \Thickness}
\kern-\Thickness}

\def\Vsquare#1{\vbox{\Square{$#1$}}\kern-\Thickness}

\title[Multivariate polynomials as $S_n$-modules]{A combinatorial model for the decomposition of multivariate polynomial rings as $S_n$-modules}
\author[Rosa Orellana]{Rosa Orellana}%
\address{Dartmouth College, Mathematics Department, Hanover, NH 03755, USA} \email{rosa.c.orellana@dartmouth.edu}%

\author[Mike Zabrocki]{Mike Zabrocki}%
\address{Department of Mathematics and Statistics, York University, Toronto, Ontario M3J 1P3,
Canada} \email{zabrocki@mathstat.yorku.ca}%

\date{}  

\thanks{Work supported by NSF grants DMS-1300512 and DMS-1700058, and by NSERC.
}

\begin{document}

\maketitle

\begin{abstract}
We consider the symmetric group $S_n$-module
of the polynomial ring with $m$ sets of $n$ commuting variables and $m'$ sets of $n$ anti-commuting
variables and show that the multiplicity of
an irreducible indexed by the partition $\lambda$ (a partition of $n$)
is the number of multiset tableaux of shape $\lambda$
satisfying certain column and row strict conditions.  We also
present a finite generating set for the ring of
$S_n$ invariant polynomials of this ring.
\end{abstract}

\tableofcontents

\begin{section}{Introduction}

Let $m$ and $n$ be positive integers.
The multivariate polynomial ring of $m$ sets of $n$ commuting variables is a
$GL_n \times GL_m$-module that is familiar in the combinatorial representation
theory literature.  Denote this module by
$$\CC[X_{n \times m}] := \CC[ x_{ij} : 1 \leq i \leq n, 1 \leq j \leq m] $$
then it is well known (e.g. \cite{GoodWall} Theorem 5.6.7) that
the space decomposes as
$$\CC[X_{n \times m}] \simeq
\bigoplus_{\lambda} W_n^\lambda \otimes W_m^\lambda$$
where the direct sum is over all partitions with length less than or equal to $\text{min}(m,n)$
and $W_n^\lambda$ is a polynomial irreducible $GL_n$-module indexed
by the partition $\lambda$.  More precisely, as a $GL_n$-module,
the multiplicity of the irreducible module $W_n^\lambda$ is
equal to the dimension of $W_m^\lambda$. This dimension is equal to the number
of column strict tableaux of shape $\lambda$ and content in the entries
$\{1,2,\ldots, m\}$.  The actions
of $GL_n$ and $GL_m$ commute with each other and this decomposition is a
consequence of the double centralizer theorem.

For a sequence $(a_1, a_2, \ldots, a_m)$ of non-negative integers
the span of the monomials such that, for each $i$ between $1$ and $m$,
the degree in the variables $x_{1i}, x_{2i}, \ldots, x_{ni}$ is equal to $a_i$
is a $GL_n$ submodule of $\CC[X_{n \times m}]$.
This homogeneous submodule has character
$h_{a_1}[X_n] h_{a_2}[X_n] \cdots h_{a_m}[X_n]$
where the $h_r[X_n]$ are the complete homogenous symmetric functions.

The symmetric group $S_n$, realized as permutation matrices,  is a subgroup of $GL_n$
and so this subspace is also a $S_n$-module.  The multiplicity of the irreducible $S_n$-module
indexed by the partition $\lambda$ in this module can be expressed in terms of
plethysm \cite{Lit, ST} of symmetric functions,
\begin{equation}
\label{eq:littlewood}
\left< h_{a_1} h_{a_2} \cdots h_{a_m}, s_\lambda[ 1 + h_1 + h_2 + \cdots] \right>~.
\end{equation}
While there are no general techniques for computing plethysm
multiplicities, it is possible to give a combinatorial
interpretation for this particular expression (e.g. \cite[Theorem 10]{LR}, \cite{LW}).

We extend the module under consideration by looking at polynomial rings in
$m$ sets of commuting variables and $m'$ sets of anticommuting
variables.  That is, let
$$\CC[X_{n \times m}; \Theta_{n \times m'}]
:= \CC[ x_{ij}, \theta_{ij'} : 1 \leq i \leq n, 1 \leq j \leq m, 1 \leq j' \leq m']$$
where the variables $x_{ij}$ commute and commute with the
$\theta_{ij'}$ variables and $\theta_{ij} \theta_{ab} = -\theta_{ab} \theta_{ij}$
if either $i \neq a$ or $j \neq b$, and $\theta_{ij}^2 = 0$.
There is a $GL_n \times GL_m \times GL_{m'}$ action
on this space; however, in this paper we concentrate on the restriction
of the $GL_n$ action to the subgroup of permutation matrices.
This copy of the symmetric group acts on the first indices
of the variables.  We are particularly interested in the decomposition of the
subspaces of fixed homogeneous degree.   In this case, it is also possible
to give an interpretation for the multiplicity of an irreducible
module in terms of plethysm.  However, in this case, we are not aware
of general techniques for finding a combinatorial interpretation for this multiplicity.
The main goal of this paper is to give a combinatorial interpretation for this
multiplicity in terms of tableaux.


In Section 2, we give definitions and introduce the notation used in this paper.  Then, in
Section 3 and Section 4 we present the two main results:
\begin{itemize}
\item  A combinatorial interpretation for the
multiplicity of an irreducible symmetric group
representation in a homogeneous component of $\CC[X_{n \times m}; \Theta_{n \times m'}]$
in terms of certain multiset tableaux (see Theorem \ref{th:main}).
\item  A finite set of algebraic generators for the $S_n$
invariants of $\CC[X_{n \times m}; \Theta_{n \times m'}]$ (see  Theorem \ref{th:gens}).
\end{itemize}

An interesting consequence of the combinatorial interpretation is that
it shows that the symmetric group submodule of
fixed homogeneous degree is representation stable in the sense
defined in \cite{CF, CEF}.

An important application of the main theorem of this paper and the
double centralizer theorem \cite{CR,GoodWall} is
to give an interpretation to the dimensions of the irreducible
representations of the centralizer of ${\mathbb C}S_n$ when it acts
on multivariate polynomial rings (see \cite{NPS,OZ3}).

\subsection*{Acknowledgements}
The authors would like to thank the referees for a careful reading of this paper and thoughtful feedback.
\end{section}

\begin{section}{Notation and Preliminaries} \label{sec:notation}
Let $X_n^{(i)}$ represent a collection of commuting
variables $x_{1i}, x_{2i}, \ldots, x_{ni}$ on which the symmetric group $S_n$ acts
by permutation of the first index.
That is, $\sigma(x_{ri}) = x_{\sigma(r)i}$ for all $\sigma \in S_n$.
The notation $\Theta_n^{(i)}$ will be used to represent a collection
of anti-commuting variables (\defn{Grassmannian variables}) $\theta_{1i}, \theta_{2i}, \ldots, \theta_{ni}$
(again, on which the symmetric group acts on the first index).
Now denote the polynomial ring in $m$ sets of the commuting variables
and $m'$ sets of anti-commuting variables by
$$\CC[X_{n \times m}; \Theta_{n \times m'}] := \CC[ X_{n}^{(1)}, X_n^{(2)}, \ldots, X_n^{(m)}
, \Theta_n^{(1)}, \Theta_n^{(2)}, \ldots, \Theta_n^{(m')}]$$
where the product satisfies the relations
\[
\theta_{ri} \theta_{sj} = - \theta_{sj} \theta_{ri} \hbox{ if } r\neq s\hbox{ or }i \neq j
\qquad\hbox{ and }\qquad
\theta_{ri}^2 =0
\]
\[
x_{rk} x_{sd} = x_{sd} x_{rk}\qquad \hbox{ and }\qquad
x_{rk} \theta_{sj} = \theta_{sj} x_{rk}
\]
for $1 \leq r,s \leq n$, $1 \leq i < j \leq m'$ and $1 \leq k \leq d \leq m$.

A monomial in $\CC[X_{n \times m}; \Theta_{n \times m'}]$ is said to be
of \defn{degree} $\alpha = (\alpha_1, \alpha_2, \ldots, \alpha_m)$ in the
commuting variables and $\beta = (\beta_1, \beta_2, \ldots, \beta_{m'})$
in the Grassmannian variables
if the total degree in the variables $X_n^{(k)}$ is $\alpha_k$ and
the total degree of the monomial in the variables $\Theta_n^{(i)}$
is $\beta_i$ for $1 \leq i \leq m'$ and $1 \leq k \leq m$.
The homogeneous subspace spanned by all monomials of degree
$\alpha = (\alpha_1, \alpha_2, \ldots, \alpha_m)$ in the
commuting variables and $\beta = (\beta_1, \beta_2, \ldots, \beta_{m'})$
in the Grassmannian variables is an $S_n$ submodule of
$\CC[X_{n \times m}; \Theta_{n \times m'}]$.

There is another notation for this symmetric group module that
is worth mentioning in terms of the symmetric tensor $S^r(V)$ and
antisymmetric tensor $\bigwedge^{r'}(V)$.  If $V_n$ is a vector space of dimension $n$
with a basis $\{v_1, v_2, \ldots, v_n\}$.  We note that as $S_n$-modules,
$$\CC[X_{n \times m}; \Theta_{n \times m'}]
\simeq \bigoplus_{r,r' \geq 0} S^r( V_n \otimes V_m ) \otimes {\bigwedge}^{r'}( V_n \otimes V_{m'} )$$
where the symmetric group $S_n$ acts on the vector space $V_n$ in this expression.

\subsection{Combinatorial definitions} \label{subsec:notation}
A \defn{partition} of an integer $n$ is a sequence of positive weakly decreasing
integers whose terms sum to $n$.  The notation $\lambda \vdash n$ denotes
that $\lambda$ is a partition of $n$ and $\ell(\la)$ denotes the
number of terms in the sequence. We use $|\lambda|$ to denote the sum of the terms
of the partition $\lambda$.  The \defn{cells} of a partition $\lambda$
is the set of pairs $\{ (i, j) : 1 \leq i \leq \ell(\lambda), 1 \leq j \leq \lambda_i \}$.
In this paper, the cells will be graphically represented by displaying them in the first
quadrant using French notation with the largest row of the partition on the bottom.

A \defn{multiset} is a collection of objects where the entries are allowed to repeat.
Multisets will be indicated by enclosing the collection of elements with
$\dcl, \dcr$ to indicate that the structure keeps the multiplicity of the
elements.  When the multiset has entries which are integers between $1$ and $m$,
the \defn{content vector} of the multiset will be a vector $(a_1, a_2, \ldots, a_m)$ where
$a_i \geq0$ is the number of times that $i$ appears in the multiset.

A \defn{multiset partition} is a multiset of multisets.
That is, $\pi = \dcl S_1, S_2, \ldots, S_r \dcr$ where
each of the $S_i$ is a multiset.  The entries in $\pi$
are referred to as the \defn{parts} of $\pi$ and the \defn{length} of
$\pi$ is the number of (non-empty) parts of $\pi$.  The
\defn{content} of a multiset partition $\pi$ is the disjoint
union of the entries of $\pi$, that is, $\biguplus_{i=1}^{\ell(\pi)} S_i$.
The multiset partitions that appear in this paper will be
in two different alphabets.  Fix two non-negative integers $m$ and $m'$.
The multiset partitions that are considered here will have entries in
$[m] \cup [\o{m'}] := \{1,2, \ldots, m\} \cup \{\o1, \o2, \ldots, \o{m'}\}$
where barred entries are allowed to occur at most once in each part of the multiset partition
(but may occur in several parts of a given multiset partition).
The notation $\pi \mvdash S$ will be used to indicate
that $\pi$ is a multiset partition with content $S$.
The condition that barred entries are not allowed to occur twice
in any given part of a multiset partition is imposed by the
algebraic relation that the Grassmannian variables square to zero
and that this data structure is used to encode these algebraic objects.
\defn{Throughout this paper it will be assumed that barred entries
may not repeat within a single multiset, a part of a multiset partition,
or within a single cell of a tableau or filling.}

Multisets are, by definition, an unordered structure, but it will be necessary
to specify an order on multisets for constructing multiset tableaux.
The results of this paper are independent of the order we choose; however, the
order chosen will need to be consistent with the algebra.  For instance,
in the proof of \ref{th:main} when multiplying monomials  and in Section \ref{sec:invariants}
the basis elements will inherit a sign from the order.
We assume that our alphabet is totally ordered by $1 < 2 < \ldots < m < \o1 < \o2 < \ldots < \o{m'}$
In this paper,  the multisets will be ordered in \defn{reverse lexicographic order}.
For multisets $S$ and $S'$, we say $S < S'$ if $max(S) < max(S')$, and if
$max(S)=max(S')$ a comparison is made by removing one instance of $max(S)$ from each multiset.
The empty multiset is considered to be the smallest multiset in this order.
(See Examples \ref{ex:tab} and \ref{ex:filling} for an illustration of the use of this multiset order).

Let $\lambda$ be a partition and $S$ be a multiset of entries from $[m] \cup [\o{m'}]$.
A \defn{multiset tableau} of shape $\lambda$ and content $S$ is a map, $T$, from the cells of $\lambda$
to multisets of entries of $[m] \cup [\o{m'}]$ satisfying the following conditions.
\begin{enumerate}
\item For each $c \in \lambda$, $T(c)$ is a multiset (which could be empty)
in barred and unbarred entries, $[m] \cup [\o{m'}]$.
\item The cells are weakly increasing in both the rows and columns with
respect to the chosen order on multisets.
That is, $T(c_1, c_2) \leq T(c_1+1,c_2)$ and $T(c_1,c_2) \leq T(c_1, c_2+1)$ whenever
both cells are in the partition.
\item If a multiset label contains an even number of barred entries,
then no two cells that are labelled with
that multiset may occur in the same column
(i.e. the cells with the same multiset label that have an even number
of barred entries form a horizontal strip).  That is, if $S_i$ is a multiset with an even
number of barred entries and $(c_1, c_2)$, $(c_1',c_2')$ are cells of $\lambda$
such that $T(c_1, c_2) = T(c_1',c_2') = S_i$, then
either $(c_1, c_2) = (c_1',c_2')$ or $c_2 \neq c_2'$.
\item If a multiset contains an odd number of barred entries, then no two cells that are labelled with
that multiset may occur in the same row
(i.e. the cells with the same multiset label that have an
odd number of barred entries form a vertical strip).
That is, if $S_i$ is a multiset with an odd
number of barred entries and $(c_1, c_2)$, $(c_1',c_2')$ are cells of $\lambda$
such that $T(c_1, c_2) = T(c_1',c_2') = S_i$, then
either $(c_1, c_2)=(c_1',c_2')$ or $c_1 \neq c_1'$.
\end{enumerate}

The \defn{content} of a tableau is the multiset union of the content of the cells of the tableau.

\begin{example}\label{ex:tab}
Let $\lambda = (7,3,2,2,1)$, then the following
is an example of a multiset tableau
of content $S = \dcl 1^3, 2^4, \o1^6, \o2^6 \dcr$
satisfying the conditions of the definition.
Since $\dcl \dcr < \dcl 2\dcr < \dcl 1, \o1\dcr < \dcl \o2\dcr < \dcl \o1, \o2\dcr~.$
\squaresize=13pt
\[
\young{1\o1\cr1\o1&\o2\cr1\o1&\o2\cr2&2&\o1\o2\cr&&2&2&\o2&\o1\o2&\o1\o2\cr}
\]
\end{example}

\subsection{Symmetric Functions} The proof of the main result will require the use of well known identities and notation in
symmetric functions.  We will mainly follow the notation which is common to references in
this area \cite{Mac, Sagan, Stanley} with a single
addition that we describe below.  The ring of symmetric functions is the polynomial
algebra in generators $p_i$ (\defn{the power sum generators})
for $i\geq 1$ where the degree $p_i$ is $i$.  That is,
\[
\Lambda = {\mathbb Q}[p_1, p_2, p_3, \ldots ]~.
\]
The \defn{elementary} $\{ e_i \}_{i \geq 1}$ and \defn{homogeneous}
generators $\{ h_i \}_{i \geq 1}$ are related to the power sum generators
by the equations

\[
n e_n = \sum_{r=1}^n (-1)^{r-1} p_r e_{n-r}
\qquad\hbox{ and }\qquad
n h_n = \sum_{r=1}^n p_r h_{n-r}~.
\]
To allow for simpler notation, let $h_0=e_0=p_0=1$ and $h_{-r} = e_{-r} = p_{-r} =0$
for $r>0$.  For an integer vector $\alpha= (\alpha_1, \alpha_2, \ldots, \alpha_{\ell(\alpha)}),$
products of the generators will be represented by the shorthand
\[
p_\alpha:= p_{\alpha_1} p_{\alpha_2} \cdots p_{\alpha_{\ell}},
\qquad
e_\alpha:= e_{\alpha_1} e_{\alpha_2} \cdots e_{\alpha_{\ell}}
\qquad\hbox{ and }\qquad
h_\alpha:= h_{\alpha_1} h_{\alpha_2} \cdots h_{\alpha_{\ell}}~.
\]

For a partition $\lambda$ of $n$, denote the
irreducible character of the representation of the symmetric group $S_n$
indexed by the partition $\lambda$ by $\chi^\lambda$ and the value of this
character at a permutation of cycle type $\mu$ by $\chi^\lambda(\mu)$.
The \defn{Schur symmetric functions} are defined as
\[
s_\lambda = \sum_{\mu} \chi^{\lambda}(\mu) \frac{p_\mu}{z_\mu}
\]
where $z_\mu = \prod_{i\geq1} m_i(\mu)! i^{m_i(\mu)}$ and $m_i(\mu)$
is the number of times that $i$ occurs in the partition $\mu$.

The \defn{Hall scalar product} on symmetric functions is defined
for the power sum basis as
\[
\left< p_\lambda, p_\mu \right> = \begin{cases}z_\mu&\hbox{ if }\lambda=\mu\\
0&\hbox{ otherwise}
\end{cases}~.
\]

There is a combinatorial rule for multiplying an elementary or homogeneous generator
and a Schur function that is known as the \defn{Pieri rule}.  It says
\begin{equation}
h_r s_\lambda = \sum_\mu s_\mu\qquad\hbox{ and }\qquad e_r s_\lambda = \sum_\gamma s_\gamma
\end{equation}
where the sum on the left is over partitions $\mu$ such that $\lambda_i \leq \mu_i$ and for
all cells $(c_1, c_2)$ and $(c_1', c_2')$ in $\mu$ which are not also in $\lambda$,
either $(c_1, c_2) = (c_1', c_2')$ or $c_2 \neq c_2'$.  The sum on the right is over partitions
$\gamma$ such that $\lambda_i \leq \gamma_i$ and for
all cells $(c_1, c_2)$ and $(c_1', c_2')$ in $\gamma$ which are not also in $\lambda$,
either $(c_1, c_2) = (c_1', c_2')$ or $c_1 \neq c_1'$.

Symmetric functions play a role both as generating functions for characters of the symmetric
group, and also as polynomial characters of $GL_n$ representations.
The \defn{character of a $GL_n$ representation} is the trace of the representation
when it is evaluated at a
diagonal matrix with eigenvalues $x_1, x_2, \ldots, x_n$.
It will always be the case that this character, as a function of the eigenvalues,
will be equal to a symmetric
function $f \in \Lambda$ where $f$ is expanded in the power sum generators
and $p_k$ is replaced by $x_1^k + x_2^k + \cdots + x_n^k$.  Denote this character by $f[X_n]$.

The symmetric group may be realized as the subgroup of permutation matrices inside of
$GL_n$.  For each permutation $\sigma$, let $A_\sigma$ represent the corresponding permutation matrix.
To compute the value of the character of an $S_n$ representation
with character equal to $f[X_n]$,
the variables of $f[X_n]$ are replaced by the eigenvalues of $A_\sigma$.  Let $\mu$
be a partition of $n$ and $\sigma$ a permutation of cycle structure $\mu$.  Up to
reordering, the eigenvalues of $A_\sigma$ are dependent only on the cycle structure
of the permutation $\sigma$ (that is, on the partition $\mu$).
We will denote the \defn{evaluation of $f[X_n]$
at the eigenvalues of a permutation matrix} of cycle structure $\mu$ by $f[\Xi_\mu]$
and this is a character value of the $S_n$ representation.

The reason that the computation of the character is important for this problem
is that the symmetric group character characterizes an $S_n$ representation up to isomorphism.
That is, let $X$ be a $GL_n$ representation with character $f[X_n]$ as a function
of the eigenvalues of a permutation matrix.
The \defn{Frobenius image of the character} \cite[Ch I.7, equation (7.2)]{Mac}, \cite[Ch 4.7]{Sagan}
is the generating function
\begin{equation}\label{eq:frobim}
\phi(f[X_n]) = \sum_{\mu \vdash n} f[\Xi_\mu] \frac{p_\mu}{z_\mu}
\end{equation}
and the multiplicity of an irreducible $S_n$ representaton indexed by the partition $\lambda$
in $X$ is equal to the coefficient of $s_\lambda$ in $\phi(f[X_n])$.

\begin{example} Note that $s_2[X_3] = x_1^2+x_2^2+x_3^2 + x_1 x_2 + x_1 x_3 + x_2 x_3$
is the character of the
symmetric group representation corresponding to the module $S^2(V_3)$.
To compute the value of the character at the permutations of cycle type $\mu = (1,1,1)$ with
eigenvalues $\{1,1,1\}$,
cycle type $(2,1)$ with eigenvalues $\{1,-1,1\}$ and cycle type $(3)$ with eigenvalues
$\{ 1, e^{2\pi i/3}, e^{4 \pi i/3} \}$.  The character values are the evaluations:
$$s_2[\Xi_{(1,1,1)}] = 6, \qquad s_2[\Xi_{(2,1)}] = 2,\qquad\hbox{ and }\qquad s_2[\Xi_{(3)}] = 0~.$$

It follows that
\[
\phi(s_2[X_3]) = 6 \frac{p_{111}}{6} + 2 \frac{p_{21}}{2} + 0 \frac{p_3}{3} = 2 s_3 + 2 s_{21}.
\]
This implies that $S^2(V_3)$ decomposes into 4 irreducible $S_n$ components.
\end{example}
\end{section}

\begin{section}{A combinatorial model for the $S_n$-decomposition of
$\CC[ X_{n \times m}; \Theta_{n \times m'}]$}
\label{sec:heart}

One of the main results of this paper is the following combinatorial model for the
decomposition of the multivariate polynomial ring as an $S_n$-module.

\begin{theorem} \label{th:main}
The multiplicity of the symmetric group irreducible indexed by the partition $\lambda \vdash n$
in the subspace of degree $\alpha = (\alpha_1, \alpha_2, \ldots, \alpha_m)$
in the commuting variables and degree $\beta = (\beta_1, \beta_2, \ldots, \beta_{m'})$ in
the Grassmannian variables is equal to the number of
multiset tableaux (see the definition in Section \ref{subsec:notation})
of content $\dcl 1^{\al_1}, 2^{\al_2}, \ldots, m^{\al_m},
\o1^{\be_1}, \ldots, \o{m'}^{\be_{m'}} \dcr$ and of shape $\lambda$.
\end{theorem}

We present an example of this theorem in Example \ref{ex:mainth} at the
end of this section.  All of the hard combinatorial effort for proving this
theorem appears in two recent references
of the authors \cite{OZ, OZ2}.  By referring the reader to the combinatorial
interpretations in those papers we can present a relatively short proof of this
result, but there is a part which is admittedly not completely self contained.

\begin{remark}
The case of $m=1$ and $m'=0$ or $m=0$ and $m'=1$ is a well known result in the theory of symmetric
functions due to A. C. Aitkin \cite {Ait1, Ait2}
(see \cite{Stanley} p.474--5 exercises {\bf 7.72} and {\bf 7.73}).
The case of $m>0$ and $m'=0$ follows from a result of Littlewood \cite{Lit,ST} and known
techniques for calculating plethysm coefficients.
The multivariate version that we present here is a repeated tensor of Aitkin's results.
What we hope to convey is the surprising fact that the decomposition of this symmetric group module
has a simple description  in terms of `multiset tableaux' and these combinatorial objects specialize to several well
known special cases.
\end{remark}

The following lemmas and propositions
involve finding combinatorial interpretations for algebraic expressions.
In particular, we use combinatorial objects to explain coefficients in
expressions of symmetric functions.
When we began extending our combinatorial
results to explain the decomposition of expressions that are no longer bases of
the symmetric functions, the multiset tableaux that appear in Theorem \ref{th:main}
were a consequence.

Before we prove the theorem we state the following lemma
which is a typical calculation of a computation of
a $GL_n$ character.

\begin{lemma}\label{lem:glnchar} The $GL_n$ character of the subspace of degree
$\alpha = (\alpha_1, \alpha_2, \ldots, \alpha_m)$
in the $x_{ij}$ variables and degree $\beta = (\beta_1, \beta_2, \ldots, \beta_{m'})$ in
the $\theta_{ik}$ Grassmannian variables is equal to $h_\alpha[X_n] e_\beta[X_n]$.
\end{lemma}

The proof will also require a combinatorial interpretation for the evaluation
of this character at eigenvalues of a permutation matrix of cycle structure
$\mu$ because we will explicitly compute the Frobenius character.  For this
we need the following combinatorial definitions.

\begin{definition} (Definition 33 of \cite{OZ})
Let $\bT_{\alpha,\mu}$ be the set of fillings of
some of the cells of the partition $\mu$ with
multisets such that the total content of the
filling is $\dcl1^{\alpha_1},2^{\alpha_2},\ldots,$ ${\ell(\alpha)}^{\alpha_{\ell(\alpha)}}\dcr$
and  any number of
labels can go into the
same cell but all cells in the same row must have
the same multiset of labels.
\end{definition}

\begin{definition} (Definition 5.13 of \cite{OZ2})
For a non-negative integer vector $\beta$ and
a partition $\mu$ let
$\oT_{\beta,\mu}$ be the fillings of some of the cells
of the diagram of the partition $\mu$ with subsets
of $\{\o1,\o2,\ldots,\o{\ell(\beta)}\}$ such that the total
content of the filling is
$\dcl\o1^{\beta_1},\o2^{\beta_2},\ldots,\o{\ell(\beta)}^{\beta_{\ell(\beta)}}\dcr$
and such that all cells in the same row have the same
subset of entries.  For $F \in \oT_{\beta,\mu}$, we define
the \defn{weight} of $F$, $wt(F)$, to be $-1$ to the power of the number of filled cells
plus the number of rows occupied by the sets of odd size.
\end{definition}

These two combinatorial definitions are used to describe the
following expressions for symmetric group characters
whose characters are given as homogeneous and elementary
symmetric functions.

\begin{prop} (Proposition 27 and Theorem 37 of \cite{OZ};
Proposition 5.6 and Lemma 5.15 of \cite{OZ2}) \label{prop:grouping}
For non-negative integer vectors $\alpha$ and $\beta$ and any
partition $\mu$,
\begin{equation}
h_\alpha[\Xi_\mu] = |\bT_{\alpha,\mu}|
\qquad\hbox{ and }\qquad
e_\beta[\Xi_\mu] = \sum_{F \in \oT_{\beta,\mu}} wt(F)~.
\end{equation}
\end{prop}

Now a product of these expressions will have terms indexed by
elements in $\bT_{\alpha,\mu} \times \oT_{\beta,\mu}$
and combining the objects into multiset fillings of $\mu$ establishes that
$$h_\alpha[\Xi_\mu]e_\beta[\Xi_\mu]$$
is equal to the sum over all fillings of the diagram for $\mu$
with multisets of content
$$\dcl1^{\alpha_1},2^{\alpha_2},\ldots,{\ell(\alpha)}^{\alpha_{\ell(\alpha)}},
\o1^{\beta_1},\o2^{\beta_2},\ldots,\o{\ell(\beta)}^{\beta_{\ell(\beta)}}\dcr$$
where the cells in any given row must have the same label, and
there is a weight of the filling equal to
$-1$ to the power of the number of filled cells
plus the number of rows occupied by the multisets with an odd number of barred entries.

Our theorem follows from one final result that we pull from \cite{OZ2} and state here without
proof.

\begin{prop} (Proposition 5.8 \cite{OZ2})\label{prop:evalhet}
For partitions $\la, \tau$ and $\mu$, let
${\mathcal F}^\mu_{\lambda,\tau}$ be the fillings of the diagram
for the partition $\mu$ with $\la_i$ labels $i$ and
$\tau_j$ labels $j'$ such that all cells in a row are filled with
the same label.  For $F \in {\mathcal F}^\mu_{\lambda,\tau}$,
the weight of the filling, $wt(F)$ is equal to $-1$ raised to the number
of cells filled with primed labels plus the number of rows occupied
by the primed labels, then
\begin{equation}
\left< h_{|\mu|-|\la| - |\tau|} h_\lambda e_\tau, p_\mu \right>
= \sum_{F \in {\mathcal F}^\mu_{\lambda,\tau}}
wt(F)~.
\end{equation}
\end{prop}

This last proposition indicates that we should assign a `type'
to each filling described above and group the fillings with the same type together.

\begin{definition}
Let $F$ be a filling of the diagram for $\mu$ with multisets such that the multiset
union of all of the labels is of content
$$cont_{\alpha,\beta} := \dcl1^{\alpha_1},2^{\alpha_2},\ldots,{\ell(\alpha)}^{\alpha_{\ell(\alpha)}},
\o1^{\beta_1},\o2^{\beta_2},\ldots,\o{\ell(\beta)}^{\beta_{\ell(\beta)}}\dcr~.$$
We associate the filling to a multiset partition, $\MSP(F)$, which is equal to the
multiset collection of the non-empty labels of the cells.

For a given multiset partition, define $\mt_e(\pi)$ to be the partition
whose entries are the multiplicities of the parts of $\pi$ that have an even number of
barred entries and $\mt_o(\pi)$ be the partition whose entries are the
multiplicities of the parts of $\pi$ that have an odd number of
barred entries.
\end{definition}

\begin{example} \label{ex:filling}
Let $n=24$ and $\mu = (5,5,3,2,2,2,2,1,1,1)$ and consider the filling $F$,
\squaresize=16pt
$$\young{112\cr22\cr\cr1\o1\o2&1\o1\o2\cr\o1&\o1\cr1\o1&1\o1\cr
\o1&\o1\cr&&\cr13&13&13&13&13\cr&&&&\cr}$$
  The content of this filling is $\dcl1^{11},2^3,3^5,\o1^8,\o2^2 \dcr$
and $\MSP(F) = \dcl\dcl1, 1, 2\dcr,$ $\dcl2,2\dcr,$
$\dcl1,3\dcr,$ $\dcl1,3\dcr,$ $\dcl1,3\dcr,$ $\dcl1, 3\dcr,$ $\dcl1, 3\dcr,$
$\dcl\o1\dcr,$ $\dcl\o1\dcr, \dcl\o1\dcr,$
$\dcl\o1\dcr,$
$\dcl1,\o1\dcr,$ $\dcl1,\o1\dcr,$
$\dcl1,\o1,\o2\dcr,$ $\dcl1,\o1,\o2\dcr \dcr$ where (while a multiset is by definition an unordered structure) we
have listed the entries by increasing reverse lexicographic order to be consistent with the
order that we will use on tableaux.
There are $6$ cells which have labels with an odd number of barred entries and they
occupy 3 rows so the weight of this filling is $-1$.
If we let $\pi = \MSP(F)$, then $\mt_e(\pi) = (5,2,1,1)$ and
$\mt_o(\pi) = (4,2)$.
\end{example}

\begin{proof}(of Theorem 3.1)
The $GL_n$ character of the subspace of degree $\alpha = (\alpha_1, \alpha_2, \ldots, \alpha_m)$
in the $x_{ij}$ variables and degree $\beta = (\beta_1, \beta_2, \ldots, \beta_{m'})$ in
the $\theta_{ik}$ Grassmannian variables is
equal to $h_\alpha[X_n] e_\beta[X_n]$ by Lemma \ref{lem:glnchar}.  We will
use the evaluation of this character at elements of the symmetric group as a subset of
elements of $GL_n$.

With the $GL_n$ character, we can compute the $S_n$ character by evaluating
$h_\alpha[X_n] e_\beta[X_n]$ at the eigenvalues of a permutation matrix.
Fix a partition $\mu$ of $n$ and we will calculate, using the combinatorial
gadgets, the character of the subspace at a permutation matrix of cycle structure $\mu$.
Proposition \ref{prop:grouping} implies that $h_\alpha[\Xi_\mu] e_\beta[\Xi_\mu]$
is equal to a sum over
fillings of the diagram for $\mu$
with multisets of content
$cont_{\alpha,\beta}$.
We then group all fillings by the associated multiset partition to the filling,
$\MSP(F)$, hence
$$h_\alpha[\Xi_\mu] e_\beta[\Xi_\mu] =
\sum_{\pi} \sum_{F : \MSP(F) = \pi} wt(F)$$
where the outer sum is over all multiset partitions $\pi$ of
such that $\pi \mvdash \dcl 1^{\alpha_1}, 2^{\alpha_2}, \ldots, m^{\alpha_m},$
$\o1^{\beta_1}, \o2^{\beta_2}, \ldots, \o{m'}^{\beta_{m'}}\dcr$.

We next apply Proposition \ref{prop:evalhet} and consider the labels of the filling
of multisets where the `primed' entries of the filling
are those with an odd number of barred entries and the `unprimed' entries of the
filling are those with an even number of barred entries.  This implies that
$$h_\alpha[\Xi_\mu] e_\beta[\Xi_\mu] =
\sum_{\pi} \left< h_{|\mu|-|\mt_e(\pi)|
-|\mt_o(\pi)|} h_{\mt_e(\pi)} e_{\mt_o(\pi)}, p_\mu \right>$$
where again the sum is over all multiset partitions $\pi$ of
content $cont_{\alpha,\beta}$.

Since we have computed the character of the subspace for each permutation of
cycle structure $\mu$, from Equation \eqref{eq:frobim} we can compute the Frobenius image
of the character of the subspace as
\begin{align}
\sum_{\mu \vdash n} h_\alpha[\Xi_\mu] e_\beta[\Xi_\mu] \frac{p_\mu}{z_\mu} &=
\sum_{\pi} \sum_{\mu \vdash n}
\left< h_{n-|\mt_e(\pi)|
-|\mt_o(\pi)|} h_{\mt_e(\pi)} e_{\mt_o(\pi)}, p_\mu \right>\frac{p_\mu}{z_\mu}\nonumber\\
&= \sum_{\pi} h_{n-|\mt_e(\pi)|
-|\mt_o(\pi)|} h_{\mt_e(\pi)} e_{\mt_o(\pi)}
\label{eq:hecomb}
\end{align}
where the sum is over all multiset partitions $\pi$ of
content $cont_{\alpha,\beta}$.

To conclude the proof of the theorem we need  to establish that the
multiplicity of a Schur function indexed by a partition $\lambda$ in this expression
agrees with the description stated in the theorem.  Here is where
the order of the multisets in the tableau plays a role in determining
the combinatorial interpretation.  Each of the generators in the
product $h_{n-|\mt_e(\pi)|
-|\mt_o(\pi)|} h_{\mt_e(\pi)} e_{\mt_o(\pi)}$ will represent the cells in
the multiset tableau which are labeled by a fixed multiset.  The generator $h_{n-|\mt_e(\pi)|
-|\mt_o(\pi)|}$ represents the blank cells, and those with an even number of barred entries are
represented by the generators in the product $h_{\mt_e(\pi)}$,
while those with an odd number of barred entries are
represented by the generators in the product $e_{\mt_o(\pi)}$.
Since the multiplication in the
ring of symmetric functions is commutative, we can choose to order these terms
with respect to the total order that we have placed on these multisets.

To determine the multiplicity of the Schur function $s_\lambda$ in this expression,
we repeatedly apply the Pieri rule.  We use tableaux to keep track of the terms in the
Schur expansion of
$$h_{n-|\mt_e(\pi)|
-|\mt_o(\pi)|} h_{\mt_e(\pi)} e_{\mt_o(\pi)},$$
where the labels of the tableaux are
the multisets represented by the $h$ or $e$-generators.
The Pieri rule implies we will record $i$ cells
in a horizontal strip for each product of an $h_i$ generator, while we will
record $i$ cells in a vertical strip for each $e_i$ generator.

We provide an example below to ensure that it is clear that
the coefficient of a Schur function $s_\lambda$ in $h_{n-|\mt_e(\pi)|
-|\mt_o(\pi)|} h_{\mt_e(\pi)} e_{\mt_o(\pi)}$ is equal to the number
of multiset tableaux whose entries are the multisets of $\pi$.  By equation \eqref{eq:hecomb},
the coefficient of $s_\lambda$ in the Frobenius image of the subspace has
multiplicity equal to the total number of multiset tableaux of shape $\lambda$
and content $cont_{\alpha,\beta}$.
\end{proof}

\begin{example} \label{ex:mainth}
Let
$\pi = \dcl\dcl1, 1, 2\dcr,$ $\dcl2,2\dcr,$
$\dcl1,3\dcr,$ $\dcl1,3\dcr,$ $\dcl1,3\dcr,$ $\dcl1, 3\dcr,$ $\dcl1, 3\dcr,$
$\dcl\o1\dcr,$ $\dcl\o1\dcr,$ $\dcl\o1\dcr,$
$\dcl\o1\dcr,$
$\dcl1,\o1\dcr,$ $\dcl1,\o1\dcr,$
$\dcl1,\o1,\o2\dcr,$ $\dcl1,\o1,\o2\dcr \dcr$
where we have $\mt_{e}(\pi) = (5,2,1,1)$ and $\mt_{o}(\pi) = (4,2)$.
This implies that $h_{24-9-6} h_{(5,2,1,1)} e_{(4,2)}$ will occur as
a summand in Equation \eqref{eq:hecomb}.

For example, to compute the tableaux with entries in
$\pi$ and of shape $\lambda = (10, 8, 5, 1)$,
we first order the generators with respect to the order on
multisets mentioned in Section \ref{subsec:notation},
$\dcl\dcr < \dcl1,1,2\dcr
< \dcl2,2\dcr
< \dcl1,3\dcr
< \dcl\o1\dcr
< \dcl1,\o1\dcr
< \dcl1,\o1,\o2\dcr.$
The number of multiset tableaux with these entries
will be the coefficient of the Schur function
in the symmetric function $h_9 h_1 h_1 h_5 e_4 e_2 h_2$.
The three tableaux with this filling are:
\squaresize=12pt
\def\tone{{\hbox{\tiny 1}}}
\def\ttwo{{\hbox{\tiny 2}}}
\def\oth{{\hbox{\tiny 13}}}
\def\rrr{{\tone\o\tone\o\ttwo}}
$$\young{\o\tone\cr\hbox{\tiny 22}&\o\tone&\tone\o\tone&\rrr&\rrr\cr\hbox{\tiny 112}&\oth&\oth&\oth&\oth&\oth&\o\tone&\tone\o\tone\cr&&&&&&&&&\o\tone\cr}
\hskip .3in
\young{\o\tone\cr\oth&\oth&\o\tone&\tone\o\tone&\rrr\cr
\hbox{\tiny 112}&\hbox{\tiny 22}&\oth&\oth&\oth&\o\tone&\tone\o\tone&\rrr\cr
&&&&&&&&&\o\tone\cr}
\hskip .3in
\young{\o\tone\cr\oth&\o\tone&\tone\o\tone&\rrr&\rrr\cr
\hbox{\tiny 112}&\hbox{\tiny 22}&\oth&\oth&\oth&\oth&\o\tone&\tone\o\tone\cr
&&&&&&&&&\o\tone\cr}~.
$$\end{example}
\end{section}

\begin{section}{The ring of $S_n$-invariants of $\CC[ X_{n \times m}; \Theta_{n \times m'}]$}
\label{sec:invariants}

Since the multiplicity of an irreducible representation
indexed by a partition $\lambda$ in the $S_n$-module
of $\CC[ X_{n \times m}; \Theta_{n \times m'}]$ is equal to the number of
multiset tableaux of shape $\lambda$,
then the multiplicity of $S_n$ invariants (the irreducible indexed by $(n)$) is
equal to the number of single row multiset tableaux.
Single row multiset tableaux are in bijection with multiset partitions
as defined in Section \ref{sec:notation} with one additional condition
imposed by the construction on tableau.

Say that $\pi$ is a \defn{super multiset partition} of $[m] \cup [\o{m'}]$ if
the parts with an odd number
of barred entries appear at most once in the multiset partition.

\begin{cor} \label{cor:invariants}
A basis for the ring of $S_n$ invariants of $\CC[X_{n \times m}; \Theta_{n \times m'}]$
is indexed by super multiset partitions of $[m] \cup [\o{m'}]$
of length less than or equal to $n$.
\end{cor}

The following special cases of these rings of invariants are
examples that we are aware of that are considered in the
algebraic combinatorics literature.

\begin{itemize}
\item If $m = 1$ and $m'=0$, then the ring of invariants
of $\CC[X_n]$ are known as symmetric polynomials and are equal to
the span of the polynomials
$Sym_n := \bigoplus_{k \geq 0}\{ p_\lambda[X_n] : \lambda \vdash k \}~.$
A basis for the ring of invariants in this case is indexed by partitions
which have length of $\lambda \leq n$.

A well known result \cite{C} states that as an $S_n$-module,
$$\CC[X_n] \simeq Sym_n \otimes \CC[X_n]/I$$
where $I$ is the ideal $\left< p_k[X_n] : 1 \leq k \leq n \right>$ and this is
equal to the ideal generated by symmetric polynomials with non-constant term.
The quotient $\CC[X_n]/I$ are often referred to as the \defn{coinvariants}
and the inverse system corresponding to that quotient
are the \defn{harmonics}.

\item If $m=2$ and $m'=0$, the quotient of $\CC[X_{n \times 2}]$
by the ideal generated by the invariants of the ring
was defined by Haiman and is known as the ring of
\defn{diagonal coinvariants} \cite{Hai94}.
A combinatorial formula for the monomial expansion of the
Frobenius characteristic of this $S_n$-module was known
as \defn{the shuffle conjecture} \cite{HHLRU, CM15}.

\item For $m>2$ and $m'=0$, F. Bergeron and L.-F.\,Pr\'eville-Ratelle \cite{Ber,BPR}
considered quotients and harmonics in multivariate polynomial
spaces and their general linear group and symmetric group characters.

\item If $m=2$ and $m'=1$, then the second author \cite{Zab19} recently
proposed the quotient $\CC[X_{n \times 2}; \Theta_n]$
by the ideal generated by the invariants as a representation
theoretic model for a generalization of the shuffle conjecture
known as \defn{the delta conjecture} \cite{HRW}.

\item If $m > 1$ and $m'=0$, then the ring of invariants
of $\CC[X_{n \times m}]$ is known as \defn{MacMahon symmetric functions}
\cite{Ges, Rosas} (MacMahon \cite{MacMahon} called these \defn{symmetric
functions in several systems of parameters}).  MacMahon indexed
the basis of this space of symmetric functions by vector partitions
instead of multiset partitions.

\item If $m=0$ and $m'=2$, then Kim and Rhoades \cite{KR} give the standard
monomial basis of the quotient of $\CC[\Theta_{n \times 2}]$ by the ideal
generated by the invariants of the ring.  They (indirectly) used the
same set of generators we define in this section for this special case.

\item If $m=1$ and $m'=1$, then the ring of invariants are known
as \defn{symmetric functions in superspace} and was studied by
Desrosiers, Lapointe and Mathieu \cite{DLM}.  There the invariants
are indexed by objects called \defn{superpartitions}, which are pairs of the
form $(\Lambda^a; \Lambda^s)$,
where $\Lambda^a$ is a strict partition and $\Lambda^s$ is a partition.
L. Solomon \cite{Sol} proved that if $\{f_1, f_2, \ldots, f_n\}$
is a set of free generators for $\CC[X_n]$, then
$\{ f_1, f_2, \ldots, f_n, d(f_1), d(f_2), \ldots, d(f_n)\}$
is a set of free
generators for the algebra of invariants of $\CC[X_n; \Theta_n]$,
where $d : \CC[X_n] \rightarrow \CC[X_n; \Theta_n]$ is an operator
on polynomials defined by $d(f) = \sum_{i=1}^n \partial_{x_i}f \theta_i $.

\item If $m=1$ and $m'\geq2$, then the ring of invariants was
studied by Alarie-V\'ezina, Lapointe and Mathieu \cite{ALM}.
In that work, the invariants are indexed by generalizations
of superpartitions.
\end{itemize}

The `super' prefix of the name super multiset partition
was borrowed from the references \cite{DLM,ALM} mentioned above,
The main result of the rest of this section is to establish
a finite list of algebraic generators for the ring of $S_n$ invariants
of $\CC[ X_{n \times m}; \Theta_{n \times m'}]$ (the analogue of the power
sums) in Theorem \ref{th:gens}.

Analogues of the elementary and
complete homogeneous generators exist and we will hint, but not explicitly
state, how to define them in terms of generating functions.  The generators
of this ring are not `free' because they will satisfy relations
coming from the Grassmannian variables.

Let $\pi = \dcl S_1, S_2, \ldots, S_{\ell(\pi)} \dcr$
be a multiset partition of length less than or equal to $n$ whose entries are in $[m] \cup [\o{m'}]$.
Furthermore let us assume that
the parts of the multisets $S_i$ are ordered in weakly increasing order
and $S_i = \dcl 1^{a_{i1}}, 2^{a_{i2}}, \ldots, m^{a_{im}},
\o{s}_{i1}, \o{s}_{i2}, \ldots, \o{s}_{i\ell_i} \dcr$.
The \defn{monomial symmetric polynomial} indexed by $\pi$ is denoted
$m_\pi$ and it is defined as the polynomial in $\CC[X_{n \times m}; \Theta_{n \times m'}]$
that equals the sum of the distinct $S_n$ orbits of
$$(X;\Theta)^{\pi} := \prod_{i=1}^{\ell(\pi)} x_{i1}^{a_{i1}} x_{i2}^{a_{i2}} \cdots
x_{im}^{a_{im}} \theta_{is_{i1}} \theta_{is_{i2}}\cdots\theta_{is_{i\ell_i}}$$
where the product follows the order on the multiset partition $\pi$ and the entries in the multiset
are in increasing order. Recall that $S_n$ acts on the first indices of the variables $x_{ij}$ and $\theta_{ij}$.
If $\ell(\pi)>n$ then the monomial symmetric polynomial is $0$.

Looking carefully at this set of symmetric group invariants, we notice that
if $S_i = S_{i+1}$ and $\ell_i$ is odd (there are an odd number of barred elements in $S_i$),
then $\sigma (i, i+1) (X;\Theta)^{\pi}=
- \sigma (X;\Theta)^{\pi}$, for $\sigma \in S_n$.  As a consequence, $m_\pi = 0$ whenever
 $\pi$ contains two equal parts with an odd number
of barred entries.  Therefore the indexing set for the monomial basis
are the super multiset partitions.

Let $S$ be a multiset and let $S = T \cup \o{T}$ where
$T = \dcl 1^{a_1}, 2^{a_2}, \ldots, m^{a_m} \dcr$ is a multiset
with entries in $[m]$, and
$\o{T} = \{ \o{s}_1, \o{s}_2, \ldots, \o{s}_k\}$ is a subset of $[\o{m'}]$.
The \defn{power sum generators} are defined by
\begin{equation}\label{eq:powergen}
p_S := \sum_{r=1}^n x_{r1}^{a_1} x_{r2}^{a_2} \cdots x_{rm}^{a_m}
\theta_{r{s}_1} \theta_{r{s}_2} \cdots \theta_{r{s}_k}~.
\end{equation}

For a multiset partition $\pi = \dcl S_1, S_2, \ldots, S_{\ell(\pi)} \dcr$,
the \defn{power symmetric polynomials} are
$$p_\pi := p_{S_1} p_{S_2} \cdots p_{S_{\ell(\pi)}}$$
where the product of the generators are in increasing order with respect
to reverse lexicographic order.

\begin{example}
Let $\pi = \dcl \dcl 1,1,\o1 \dcr, \dcl 1,1,\o2\dcr, \dcl\o1,\o2\dcr, \dcl\o1,\o2\dcr\dcr$.
As long as $n \geq 4$,
$$m_\pi = \sum_{(a,b,c,d)} x_{a1}^2 \theta_{a1} x_{b1}^2 \theta_{b2} \theta_{c1} \theta_{c2} \theta_{d1}\theta_{d2}$$
where the sum is over all sequences $(a,b,c,d)$ of distinct entries with $c<d$ (since
$\theta_{c1} \theta_{c2} \theta_{d1}\theta_{d2} = \theta_{d1} \theta_{d2} \theta_{c1}\theta_{c2}$).
The following power sum generators are
\begin{align*}
p_{\dcl1,1,\o1\dcr} &= x_{11}^2 \theta_{11} + x_{21}^2 \theta_{21} + \cdots + x_{n1}^2 \theta_{n1}\\
p_{\dcl1,1,\o2\dcr} &= x_{11}^2 \theta_{12} + x_{21}^2 \theta_{22} + \cdots + x_{n1}^2 \theta_{n2}\\
p_{\dcl\o1,\o2\dcr} &= \theta_{11} \theta_{12} + \theta_{21} \theta_{22} + \cdots + \theta_{n1} \theta_{n2}~.
\end{align*}
Some explicit expansion of the products of the power sum generators shows
that
$$p_{\dcl1,1,\o1\dcr}p_{\dcl1,1,\o2\dcr}p_{\dcl\o1,\o2\dcr}^2 = m_{\dcl \dcl 1,1,\o1 \dcr, \dcl 1,1,\o2\dcr, \dcl\o1,\o2\dcr, \dcl\o1,\o2\dcr\dcr} + m_{\dcl\dcl1,1,1,1,\o1,\o2\dcr,\dcl\o1, \o2\dcr, \dcl\o1, \o2\dcr\dcr}~.$$
\end{example}

\begin{lemma} \label{lem:span} The power symmetric polynomials $\{ p_\pi \}$ are a basis for the
ring of $S_n$ invariants of $\CC[X_{n \times m}; \Theta_{n \times m'}]$
where the $\pi$ run over all super multiset partitions with $\ell(\pi) \leq n$.
\end{lemma}

\begin{proof}
The ring of $S_n$ invariants of $\CC[X_{n \times m}; \Theta_{n \times m'}]$
are clearly spanned by the monomial symmetric polynomials since they are the
$S_n$ orbits of a single monomial in the polynomial ring.  This basis
is indexed by the super multiset partitions with length less than or equal to $n$.

Consider the power sum symmetric function indexed by a super multiset partition
and order the monomials using reverse lexicographic order where the Grassmannian variables
are larger than the commutative variables.  It follows
that $p_\pi = c_\pi m_\pi$ plus terms which are smaller with respect to this order.
Therefore the set $\{ p_\pi \}$ with
$\ell(\pi) \leq n$ also spans the same space and is linearly independent.
\end{proof}

Let ${\bf q} = q_1, q_2, \ldots, q_m$ be a commuting set of variables
and ${\bf z} = z_1, z_2, \ldots, z_{m'}$ be an anticommuting set
of variables.  Define the following generating functions for
the generators
\begin{equation}\label{eq:elemgf}
E({\bf q}, {\bf z}) = \prod_{i=1}^n \left(1 + \sum_{j=1}^m q_j x_{ij}
+ \sum_{j'=1}^{m'} z_{j'} \theta_{ij'}\right)\hbox{ and }
\end{equation}
\begin{equation}\label{eq:powergf}
P({\bf q}, {\bf z}) =
-\sum_{i=1}^n \log\left( 1 - \left(\sum_{j=1}^m q_j x_{ij}
+ \sum_{j'=1}^{m'} z_{j'} \theta_{ij'}\right)\right)~.
\end{equation}
It is easily checked that these generating functions are related by
$E({\bf q}, {\bf z}) = \exp\left( -P(-{\bf q}, -{\bf z})\right)$.

The following lemma involves a standard calculation on the generating
function using the definitions in Equation \eqref{eq:powergen}
and the expansions of the expression in Equation \eqref{eq:powergf}.
The variables satisfy $z_{j'}^2 = \theta_{ij'}^2 = 0$ and this imposes the condition
that the barred entries may not be repeated in the multiset.
There is no sign introduced in the expression
from the $\theta_{ij'}$ variables because it is cancelled by the
sign from the $z_{j'}$ variables.
\begin{lemma}
Let $S = T \cup \o{T}$ where
$T = \dcl 1^{a_1}, 2^{a_2}, \ldots, m^{a_m} \dcr$ is a multiset
with entries in $[m]$, and
$\o{T} = \{ \o{s}_1, \o{s}_2, \ldots, \o{s}_k\}$ is a subset of $[\o{m'}]$.
The coefficient of $q_1^{a_1} q_2^{a_2} \cdots q_m^{a_m} z_{s_1} z_{s_2} \cdots z_{s_k}$
in $P({\bf q}, {\bf z})$ is equal to $\frac{1}{|S|} \pchoose{|S|}{a_1, a_2, \ldots, a_m, 1^k}p_S$.
\end{lemma}

Using that expression, we can obtain the expansion of $E({\bf q}, {\bf z})$ by calculating
\begin{align}
E({\bf q}, {\bf z})&= \exp\left(-P(-{\bf q}, -{\bf z})\right)\nonumber\\
 &=\sum_{n\geq0} \frac{1}{n!}\left(\sum_{S} \frac{(-1)^{|S|+1}}{|S|}
\pchoose{|S|}{a_1, a_2, \ldots, a_m, 1^k}q_1^{a_1} q_2^{a_2} \cdots q_m^{a_m} z_{s_1} z_{s_2} \cdots z_{s_k}
p_S\right)^n\nonumber\\
&= \sum_S q_1^{a_1} q_2^{a_2} \cdots q_m^{a_m} z_{s_1} z_{s_2} \cdots z_{s_k} \sum_{\pi \mvdash S} (-1)^{|S|+\ell(\pi)} a_\pi p_\pi\label{eq:epowerexp}
\end{align}
where the sum is over multisets
$S = \dcl 1^{a_1}, 2^{a_2}, \ldots, m^{a_m}, \o{s}_1, \ldots, \o{s}_k\dcr$
and the inner sum is over all super multiset partitions $\pi$ of content $S$.
The reverse lexicographic order on multiset partitions was chosen so that the
parts of the multiset partition $\pi$ will have the same order
as the entries $s_1 < s_2 < \ldots < s_k$ so as not to introduce an
additional sign in the equality at Equation \eqref{eq:epowerexp}.
For a multiset of this form, let $c(S) = \pchoose{|S|}{a_1, a_2, \ldots, a_m, 1^k}$.
If the multiset partition $\pi$ is denoted with
its multiplicities as $\pi = \dcl S_1^{m_1}, S_2^{m_2}, \ldots, S_r^{m_r} \dcr$, then
the coefficient
$$a_\pi = \frac{c(S_1)^{m_1} c(S_2)^{m_2} \cdots c(S_r)^{m_r}}{
|S_1|^{m_1} |S_2|^{m_2} \cdots |S_r|^{m_r}\cdot m_1! m_2! \cdots m_r!}~.$$

The coefficients of a monomial $q_1^{a_1} q_2^{a_2} \cdots q_m^{a_m} z_{s_1} z_{s_2} \cdots z_{s_k}$
in Equation
\eqref{eq:epowerexp} give the expansion of the elementary generator in the power sum
symmetric polynomial.

In the notation of the following theorem
$S = T \cup \o{T}$ and $S'=T' \cup \o{T}'$ are two multisets with
$T, T'$ are both multisets of $[m]$ and $\o{T}, \o{T}'$ are both subsets of
$[\o{m'}]$.

Lemma \ref{lem:span} establishes that products of the power sum generators
will span the space of invariants.  The next result states that we
only need the power sum generators of degree less than or equal to $n$
to generate the space of invariants.

\begin{theorem}\label{th:gens} The set $\{ p_S \}$, running over all possible
multisets $S$ and $|S| \leq n$, is a generating set
for the ring of invariants of $\CC[X_{n \times m}; \Theta_{n \times m'}]$.
\end{theorem}

\begin{proof}
Observe that the generating function expression in
Equation \eqref{eq:elemgf} is a polynomial, so the coefficient from Equation \eqref{eq:epowerexp}
in $E({\bf q}, {\bf z})$ is equal to $0$ if $a_1+a_2+\cdots+a_m+k>n$.
This implies that for all $S = \dcl 1^{a_1}, 2^{a_2}, \ldots, m^{a_m}, \o{s}_1, \ldots, \o{s}_k\dcr$
with $|S|>n$, then
\begin{equation}\label{eq:notfree1}
p_S =  - \sum_{\substack{\pi \mvdash S\\\pi \neq \dcl S \dcr}} (-1)^{|S|+\ell(\pi)} a_\pi p_\pi~.
\end{equation}

This implies that any $p_\pi$ such that $\pi$ contains
a part $S \in \pi$ with $|S| > n$ can be expressed in terms of
$p_\pi$ with all parts smaller than or equal to $n$ by repeatedly
applying this relation.
\end{proof}

These generators are not free however and also satisfy the relation
$$p_S p_{S'} = (-1)^{|\o{T}| \cdot |\o{T}'|} p_{S'} p_S$$
if $S$ and $S'$ are not equal and $p_S^2 = 0$ if $|\o{T}|$ is odd.

\end{section}


\begin{thebibliography}{99}

\bibitem[Ait1]{Ait1} A. C. Aitken,
\textit{On induced permutation matrices and the symmetric group},
Proc. Edinburgh Math. Soc. vol 5, issue 1, (1937), 1--13.

\bibitem[Ait2]{Ait2} A. C. Aitken,
\textit{On compound permutation matrices},
Proc. Edinburgh Math. Soc. vol 7, issue 4, (1946), 196--203.


\bibitem[ALM]{ALM} L. Alarie-V\'ezina, L. Lapointe, P. Mathieu,
\textit{$N>=2$ symmetric superpolynomials},
Journal of Mathematical Physics 58, 033503 (2017).


\bibitem[Ber]{Ber} F. Bergeron
\textit{Multivariate Diagonal Coinvariant Spaces for Complex Reflection Groups},
Advances in Mathematics, 239 (2013), 97--108.

\bibitem[BPR]{BPR}
F.\,Bergeron, and L.-F.\,Pr\'eville-Ratelle,
\textit{Higher Trivariate Diagonal Harmonics via generalized Tamari Posets},
Journal of Combinatorics, {\bf{3}} (2012), 317--341.


\bibitem[CM15]{CM15} E. Carlsson and A. Mellit.
\textit{A proof of the shuffle conjecture},
{\tt arXiv:1508.06239}, August 2015.

\bibitem[C]{C} C. Chevalley, Invariants of finite groups generated by
reflections, Amer. J. Math. 77 (1955), 778-782.

\bibitem[CF]{CF} T. Church, B. Farb,
\textit{Representation theory and homological stability},
Advances in Mathematics,
Volume 245, 1 (2013), Pages 250--314.

\bibitem[CEF]{CEF}
T. Church, J. Ellenberg, B. Farb,
\textit{FI-modules and stability for
representations of symmetric groups},
Duke Math. J.,
Volume 164, Number 9 (2015), 1833--1910.

\bibitem[CR]{CR} C. Curtis and I. Reiner,
\textit{Methods of Representation Theory: With Applications to Finite Groups
and Orders}, Vol. I, Wiley, New York, 1981.

\bibitem[DLM]{DLM} P. Desrosiers, L. Lapointe, P. Mathieu,
\textit{Classical symmetric functions in superspace},
J Algebr Comb 24 (2006) 209--238

\bibitem[GoodWall]{GoodWall} R. Goodman, N. R. Wallach,
\textit{Symmetry, Representations, and Invariants}, Springer, 2009.

\bibitem[HRW]{HRW} J. Haglund, J. Remmel, A. Wilson,
\textit{The Delta Conjecture}, {\tt arXiv:1509.07058v2}~.

\bibitem[HHLRU05]{HHLRU}
J.~Haglund, M.~Haiman, N.~Loehr, J.~B. Remmel, and A.~Ulyanov, \textit{A
combinatorial formula for the character of the diagonal coinvariants}, Duke
J. Math. \textbf{126} (2005), 195--232.

\bibitem[Hai94]{Hai94} M. Haiman, \textit{Conjectures On The Quotient Ring By
Diagonal Invariants}, J. Algebraic Combin, Volume 3,  Issue 1 (1994), 17--76.

\bibitem[GR]{GR} A. Garsia and J. B. Remmel:
\textit{Shuffles of permutations
and the Kronecker product,}
Graphs and Combinatorics, 1 (1985), pp. 217--263.

\bibitem[Ges]{Ges} I. M. Gessel,
\textit{Enumerative applications of symmetric functions},
``Actes $17^{e}$ S\'eminaire Lotharingien,'' Publ.
I.R.M.A. Strasbourg, 348, 5--17, 1988.

\bibitem[KR]{KR} J. Kim, B. Rhoades,
\textit{Lefschetz theory for exterior algebras and fermionic diagonal coinvariants},
{\tt https://arxiv.org/abs/2003.10031}~.

\bibitem[Lit]{Lit} D. E. Littlewood,
\textit{Products and Plethysms of Characters with Orthogonal,
Symplectic and Symmetric Groups},
Canad. J. Math., {\bf 10}, 1958, 17--32.

\bibitem[LR]{LR} N. Loehr, J. B. Remmel,
\textit{A computational and combinatorial expos\'e of plethystic calculus},
Journal of Algebraic Combinatorics,
March 2011, Volume 33, Issue 2, pp. 163--198.

\bibitem[LW]{LW} N. Loehr, G. Warrington,
\textit{Quasisymmetric expansions of Schur-function plethysms}, Proc.
Amer. Math. Soc. 140 (2012), 1159--1171.

\bibitem[Mac]{Mac} I.~G.~Macdonald,
\newblock \textit{Symmetric Functions and Hall Polynomials},
\newblock Second Edition, Oxford University Press,
second edition, 1995.

\bibitem[MacM]{MacMahon} P.~A.~MacMahon,
\textit{Combinatory Analysis, vol 1 and 2},
Cambridge, (1915) and (1916).  Reprinted Chelsea, New York (1960).

\bibitem[NPS]{NPS} S. Narayanan, D. Paul, S. Srivastava,
\textit{The Multiset Partition Algebra},
{\tt arXiv:1903.10809}.

\bibitem[OZ]{OZ} R. Orellana, M. Zabrocki,
\textit{Characters of the symmetric group as symmetric functions},
{\tt arXiv:1605.06672}.

\bibitem[OZ2]{OZ2} R. Orellana, M. Zabrocki,
\textit{The Hopf structure of symmetric group characters
as symmetric functions},
{\tt arXiv:1901.00378}.

\bibitem[OZ3]{OZ3} R. Orellana, M. Zabrocki,
\textit{Howe duality of the symmetric group and a multiset partition algebra}, preprint.

\bibitem[Ros]{Rosas} M. Rosas,
\textit{A combinatorial overview of the theory of MacMahon
symmetric functions
and a study of the Kronecker product of Schur functions},
Ph.D. Thesis, Brandeis University, 2000.

\bibitem[Sag]{Sagan} B. Sagan, \textit{The symmetric group}. Representations, combinatorial algorithms, and
 symmetric functions, 2nd edition,  Graduate Text in Mathematics 203.  Springer-Verlag, 2001.
 xvi+238 pp.

\bibitem[sage]{sage} W.\thinspace{}A. Stein et~al.
\newblock \textit{{S}age {M}athematics {S}oftware
({V}ersion 6.10)},
The Sage Development Team, 2016, {\tt http://www.sagemath.org}.

\bibitem[sage-combinat]{sage-co}
The {S}age-{C}ombinat community.
\newblock  \textit{{{S}age-{C}ombinat}:
enhancing Sage as a toolbox for
computer exploration in algebraic combinatorics},
{{\tt http://combinat.sagemath.org}}, 2008.

\bibitem[ST]{ST} T. Scharf, J. Y. Thibon,
\textit{A Hopf-algebra approach to inner plethysm}.
Adv. in Math. 104 (1994), pp. 30--58.

\bibitem[Sol]{Sol} L. Solomon,
\textit{Invariants of finite reflection groups},
Nagoya J. Math., 22 (1963), 57--64

\bibitem[Sta]{Stanley} R.~Stanley,
\textit{Enumerative Combinatorics, Vol. ~2},
Cambridge University Press, 1999.

\bibitem[Zab19]{Zab19} M. Zabrocki,
\textit{A module for the Delta conjecture},
{\tt arXiv:1902.08966}

\end{thebibliography}
\end{document}